\documentclass[reqno]{amsart}
\usepackage[scale=0.75, centering, headheight=14pt]{geometry}
\usepackage[latin1]{inputenc}
\usepackage[T1]{fontenc}
\usepackage{lmodern}
\usepackage[english]{babel}
\usepackage{amsmath}

\usepackage[
    type={CC},
    modifier={by},
    version={4.0},
]{doclicense}

\usepackage{amsmath,amssymb,amsfonts,amsthm}
\usepackage{mathtools,accents}
\usepackage{mathrsfs}
\usepackage{xfrac}
\usepackage{array} 
\usepackage{aliascnt}
\usepackage{booktabs} 
\usepackage{array} 

\usepackage{verbatim} 
\usepackage{subfig} 

\usepackage{mathrsfs, dsfont}
\usepackage{amssymb}
\usepackage{amsthm}
\usepackage{amsmath,amsfonts,amssymb,esint}
\usepackage{graphics,color}
\usepackage{enumerate}
\usepackage{mathtools,centernot}

\usepackage{microtype}
\usepackage{paralist} 
\usepackage{cases}
\usepackage[initials]{amsrefs}
\allowdisplaybreaks

\usepackage{braket}
\usepackage{bm}

\usepackage[citecolor=blue,colorlinks]{hyperref}
\addto\extrasenglish{}

\usepackage{enumerate}
\usepackage{xcolor}
\usepackage{aliascnt}

\makeatletter
\def\newaliasedtheorem#1[#2]#3{
  \newaliascnt{#1@alt}{#2}
  \newtheorem{#1}[#1@alt]{#3}
  \expandafter\newcommand\csname #1@altname\endcsname{#3}
}
\makeatother

\usepackage{hyperxmp}

\usepackage{indentfirst}

\usepackage{esint}
\usepackage{mathrsfs}
\usepackage{stmaryrd}

\DeclareUnicodeCharacter{00A0}{ }
\DeclareUnicodeCharacter{00A0}{~}

\makeatletter
\newsavebox{\measure@tikzpicture}

\renewcommand\labelenumi{(\roman{enumi})}
\renewcommand\theenumi\labelenumi

\newtheorem{theorem}{\bf Theorem}[section]
\newtheorem{remark}[theorem]{\bf{Remark}}
\newtheorem{definition}[theorem]{\bf Definition}
\newtheorem{lemma}[theorem]{\bf Lemma}
\newtheorem{proposition}[theorem]{\bf Proposition}
\newtheorem{corollary}[theorem]{\bf Corollary}

\newcommand{\R}{\mathbb R}
\newcommand{\N}{\mathbb N}

\newcommand{\rank}[1]{\operatorname{rank}(#1)}

\newcommand{\dist}[2]{\operatorname{dist} \left( #1, #2 \right)}

\newcommand{\cofactor}{\operatorname{cof}}

\DeclareMathOperator{\loc}{loc}
\DeclareMathOperator{\Lip}{Lip}

\allowdisplaybreaks
\numberwithin{equation}{section}
\title{$T_5$ configurations and hyperbolic systems}

\author[C. J. P. Johansson and  R. Tione]{Carl Johan Peter Johansson \and Riccardo Tione}

\address{Carl Johan Peter Johansson  
\hfill\break  EPFL SB, Station 8, CH-1015 Lausanne, Switzerland}
\email{carl.johansson@epfl.ch}

\address{Riccardo Tione
\hfill\break Max Planck Institute for Mathematics in the Sciences, Inselstrasse 22, 04103 Leipzig, Germany}
\email{riccardo.tione@mis.mpg.de}

\thanks{Electronic version of an article published in Communications of Contemporary Mathematics \url{https://doi.org/10.1142/S021919972250081X} \copyright World Scientific Publishing Company \url{https://www.worldscientific.com/worldscinet/ccm} }

\begin{document}
\doclicenseThis
\maketitle

\begin{abstract}
In this paper we study the rank-one convex hull of a differential inclusion associated to entropy solutions of a hyperbolic system of conservation laws. This was introduced in \cite[Section 7]{KMS}, and many of its properties have already been shown in \cite{LP2,LP1}. In particular, in \cite{LP1} it is shown that the differential inclusion does not contain any $T_4$ configurations. Here we continue that study by showing that the differential inclusion does not contain $T_5$ configurations.
\end{abstract}

\section{Introduction}

We study compactness properties of weak Lipschitz solutions $\varphi$ to the equation
\begin{equation}\label{conslaw}
\partial_{tt}\varphi - \partial_x(a(\partial_x\varphi)) = 0, \text{ in the open and bounded set $\Omega \subset \R^2$}, 
\end{equation}
which additionally satisfy
\begin{equation}\label{entropy}
\partial_{t}(\eta(\partial_t\varphi, \partial_x \varphi)) - \partial_x(q(\partial_t\varphi,\partial_x\varphi)) \le 0,
\end{equation}
for
\[
\eta(x,y) = \frac{1}{2}x^2 + F(y), \quad q(x,y) = a(y)x.  
\]
Here, $a$ is a $C^1$, strictly convex and monotone increasing function on the real line, and $F(y) = \int_0^ya(\xi)d\xi$. Inequality \eqref{entropy} is called an \emph{entropy} for the equation \eqref{conslaw}. It can be checked that, if $\varphi$ is smooth, then \eqref{entropy} is automatically satisfied (with an equality), but it needs to be imposed for weak solutions. This problem has its origins in the classical paper \cite{DIP} by R. J. DiPerna. In \cite[Section 7]{KMS}, B. Kirchheim, S. M\"uller and V. \v Sver\'ak propose the study of \eqref{conslaw}-\eqref{entropy} via differential inclusions. In other words, introducing auxiliary functions
\[
u := \partial_t\varphi, \quad v :=\partial_x\varphi,
\]
and assuming equality in \eqref{entropy}, we can rewrite \eqref{conslaw}-\eqref{entropy} as the system
\begin{equation}\label{sysprediff}
\begin{cases}
\partial_tu - \partial_x (a(v)) = 0,\\
\partial_tv - \partial_xu = 0,\\
\partial_t(\eta(u,v)) - \partial_x (q(u,v)) = 0.
\end{cases}
\end{equation}
At least locally, this is equivalent to asking that there exist Lipschitz functions $\varphi_1,\varphi_2$ such that the Lipschitz map $\psi = (\varphi_1,\varphi,\varphi_2)$ solves, a.e. in $\Omega$,
\[
D\psi = (\partial_t\psi,\partial_x\psi) \in K_a := \left\{\left[\begin{array}{cc} a(v) & u\\ u & v\\  a(v)u & \frac{1}{2}u^2 + F(v) \end{array}\right]: u,v \in \R^2\right\}.
\]
Rewriting a system of PDEs as a differential inclusions allows to prove compactness properties or to use convex integration methods to show unexpected features of the original system. Given a set $K \subset \R^{m\times n}$, we have \emph{compactness for the associated differential inclusion} if, given an open and bounded set $\Omega \subset \R^n$, a sequence $u_j \in \Lip(\Omega,\R^m)$ with equibounded Lipschitz norm and weakly-* converging to $u$ satisfying
\[
\dist{Du_j}{K} \to 0 \text{ in }L^1_{\loc},
\]
then $u_j$ converges strongly in $W_{\loc}^{1,1}$ to $u$. This is related to the classification of \emph{gradient Young measures } supported in $K$. In particular, compactness for the differential inclusion associated to $K$ holds if and only if the only Young measures supported in $K$ are Dirac deltas. We refer the reader to \cite[Section 4]{PER} and references therein for additional information. Usually, showing compactness of differential inclusions is a hard task, and few general methods are known, see for instance \cite[Section 3]{SAP}. 
\\
\\
Conversely, in order to disprove compactness for the differential inclusion, there are some natural steps that need to be taken. First, one needs to study rank-one connections. A rank-one connection in $K$ is a couple of matrices $\{A,B\}$ such that $A,B \in K$ and $\rank{A-B} = 1$. If there exists a rank-one connection in $K$, then it is possible to construct a highly oscillatory approximate solution to the differential inclusion, see for instance \cite[Lemma 3.2]{KIRK}, thus disproving compactness. For the particular case of $K_a$, with $a$ strictly convex and monotone increasing, it has been shown in \cite[Proposition 4]{LP1} that there are can be no rank-one connections inside $K_a$. The next step is to see whether $K$ contains $T_N$ configurations. We will postpone the precise definition of $T_N$ configuration to Section \ref{PrelTN}. $T_N$ configurations in $\R^{m\times n}$ are ordered sets of $N$ distinct matrices $X_1,\dots, X_N$ such that possibly
\[
\rank{X_i - X_j} \ge 2, \quad\forall i \neq j,
\]
but $\{X_1,\dots, X_N\}$ supports a non-trivial gradient Young measure, see \cite[Lemma 2.6]{DMU}. These special sets of matrices, found independently by many authors, see \cite[Section 2.5]{DMU}, have been the fundamental building block in the convex integration construction of the highly irregular critical points of quasiconvex functionals \cite{SMVS} and  polyconvex functionals \cite{LSP}. Since then, they never ceased to be a fundamental tool to construct pathological solutions to PDEs, compare \cite{COR,ST,GMTDI,HRT,FS}.
\\
\\
Concerning the particular differential inclusion associated to $K_a$, Kirchheim, M\"uller and \v Sver\'ak proposed in \cite{KMS} the study of four properties $(P1)-(P4)$ which concern \emph{local} compactness properties of the differential inclusion. To the best of our knowledge, the study of these properties has started in \cite{LP2}. Subsequently, A. Lorent and G. Peng started the study of \emph{non-local} compactness properties of the differential inclusion in \cite{LP1}, showing that the differential inclusion associated to $K_a$ does not contain rank-one connections and $T_4$ configurations. Notice that in $\R^{m\times n}$ for $n \neq 2$ or $m \neq 2$ a $T_N$ configuration does not need to contain\footnote{The situation is different in $\R^{2\times 2}$, where every $T_N$ configuration contains a $T_4$ configuration by \cite{SzeLas}.} a $T_4$ configuration. For instance in \cite{LSP} the author finds a $T_5$ configuration in a set $K_F$ that, by \cite[Proposition 9]{KMS}, cannot contain a $T_4$ configuration. Thus, even after the result of \cite{LP1}, it is unclear whether $K_a$ contains $T_N$ configurations for $N \ge 5$. Aim of this paper is to show that $K_a$ does not contain $T_5$ configurations. Our proof exploits the results of \cite{SzeLas} and is close in spirit to those of \cite{GMTDI,HRT}. Our method yields also a short, independent proof of the non-existence of $T_4$ configurations inside $K_a$, i.e. the main result of \cite{LP1}. We end this introduction by stating our two main results:

\begin{theorem}\label{teo1}
Assume $a \in C^1(\R)$, $a(0) = 0$ and $a$ is strictly convex and increasing. Let $$K_a = \left\{\left[\begin{array}{cc} a(v) & u\\ u & v\\  a(v)u & \frac{1}{2}u^2 + F(v) \end{array}\right]: u,v \in \R^2\right\},$$
where $F(v) = \int_0^va(\xi)d\xi$. Then, $K_a$ does not contain $T_4$ configurations.
\end{theorem}

\begin{theorem}\label{teo2}
Assume $a \in C^1(\R)$, $a(0) = 0$ and $a$ is strictly convex and increasing. Let $K_a$ be defined as in Theorem \ref{teo1}. Then, $K_a$ does not contain $T_5$ configurations.
\end{theorem}

The paper is organized as follows. In Section \ref{PrelTN}, we recall the definition of $T_N$ configuration and a characterization coming from \cite{SzeLas,GMTDI}. In Section \ref{PropKa}, we show some useful inequalities on minors of elements of $K_a$, that will be used in Section \ref{TheEnd} to show our two main Theorems \ref{teo1}-\ref{teo2}.

\subsection{Notation}\label{PrelTN}
The set of matrices with $m$ rows and $n$ columns will be denoted by $\R^{m \times n}$. The Hilbert-Schmidt scalar product in $\R^{m \times n}$ is denoted by $\langle \cdot , \cdot \rangle$. For any matrix $A \in \R^{m \times m}$, we define $\cofactor{A}$ as the $m \times m$ matrix whose value in the $i$-th row and $j$-th column is
\[
 (\cofactor{A})_{ij} = (- 1)^{i+j} \det (A^{j,i})
\]
where $A^{j,i}$ denotes the $(m-1) \times (m-1)$ matrix obtained by removing the $i$-th column and the $j$-th row from $A$. The set of symmetric matrices in $\R^{m \times m}$ is denoted by $\R^{m \times m}_{\text{sym}}$. For two symmetric matrices $A$ and $B$, we will write $A \le B$ if $B - A$ is positive semi-definite. Finally, given vectors $a_1,\dots, a_n$ for some $n \ge 1$, we set
\[
\sum_{k = i}^j a_k := 0 \quad \text{ whenever } j < i.
\]

\section*{Acknowledgement}
The first author was supported by the SNSF Grant 182565 and by the Swiss State Secretariat for Education, Research and Innovation (SERI) under the contract number M822.00034. The second author wishes to thank Prof. L\'aszl\'o Sz\'ekelyhidi for introducing him to this problem.


\section{Preliminaries on $T_N$ configurations}\label{PrelTN}
In this section we recall a useful characterization of $T_N$ configurations. This was first discovered by L. Sz{\'{e}}kelyhidi in \cite{SzeLas} for $T_N$ configurations contained in $\R^{2 \times 2}$ and generalised by C. De Lellis, G. De Philippis, B. Kirchheim and the second named author in \cite{GMTDI} to study $T_N$ configurations contained in $\R^{m \times n}$ for arbitrary integers $m$ and $n$.

\begin{definition}\label{deftn}
An ordered set of $N$ distinct matrices $\{ X_1, \ldots, X_N \}$ is said to induce a $T_N$ configuration if there exist matrices $Q, D_1, \ldots, D_N \in \R^{m \times n}$ and real numbers $k_i > 1$ such that 
\begin{enumerate}
 \item  $\rank{D_i} \leq 1$ for each $i$;
 \item the following \emph{closing condition} is fulfilled
 \[
  \sum_{i = 1}^N D_i = 0;
 \]
 \item it holds
 \begin{equation}\label{eq:TNEquations}
 X_i = Q + \sum_{j = 1}^{i-1}D_j + k_iD_i, \quad \forall 1 \le i \le N.
 \end{equation}
\end{enumerate}
\end{definition}

Let us introduce some notation. Let $I$ and $J$ be multi-indices. More precisely, we will use $I$ to denote multi-indices referring to ordered sets of rows of matrices and $J$ for multi-indices referring to ordered sets of columns of matrices. In this section, we deal with matrices in $\R^{m \times n}$ and we will therefore have
\begin{align*}
I &= (i_1, \ldots, i_r) \in \N^r, 1 \leq i_1 < \ldots < i_r \leq m \\
J &= (j_1, \ldots, j_s) \in \N^s, 1 \leq i_1 < \ldots < i_s \leq n
\end{align*}
We will use the notation $|I| = r$ and $|J| = s$. In this paper, we will always have $r = s$.
We define 
\[
\mathcal{A}_r = \{ (I, J) : \text{$I$ and $J$ are ordered sets and } |I| = |J| = r\}.
\]
For any matrix $M \in \R^{m \times n}$ and any $Z \in \mathcal{A}_r$, we write $M^Z$ to denote the $r \times r$ square matrix obtained by considering just the elements $M_{ij}$ where $(i,j) \in I \times J$.
Given a set of matrices $\{ X_1, \ldots , X_N \} \subset \R^{m \times n}$, $\mu \in \R$ and $Z \in \mathcal{A}_r$, we introduce the matrix 
\[
 A^{\mu}_Z = 
 \begin{bmatrix}
 0 & \det (X_2^Z - X_1^Z) & \det (X_3^Z - X_1^Z) & \cdots & \det (X_N^Z - X_1^Z) \\
 \mu \det (X_1^Z - X_2^Z) & 0 & \det (X_3^Z - X_2^Z) & \cdots & \det (X_N^Z - X_2^Z) \\
 \vdots & \vdots & \vdots & \ddots & \vdots \\
 \mu \det (X_1^Z - X_N^Z) & \mu \det (X_2^Z - X_N^Z) & \mu \det (X_3^Z - X_N^Z) & \cdots & 0\\
 \end{bmatrix}.
\]

The following proposition provides us with a characterisation of $T_N$ configurations.

\begin{proposition}\label{prop:CharacterizationTN}
A set $\{ X_1, \ldots, X_N \} \subset \R^{m \times n}$ induces a $T_N$ configuration if and only if there is a real $\mu > 1$ and a vector $\lambda \in \R^N$ with positive components such that
\[
 A_Z^{\mu} \lambda = 0 \quad \forall Z \in \mathcal{A}_2.
\]
\end{proposition}

\begin{corollary}
Let $\{ X_1, \ldots, X_N \} \subset \R^{m \times n}$ induce a $T_N$ configuration and let $\lambda \in \R^N$ be a vector with positive components given by Proposition~\ref{prop:CharacterizationTN} such that 
\[
 A_Z^{\mu} \lambda = 0 \quad \forall Z \in \mathcal{A}_2.
\]
Define the vectors $t^i$ as 
\begin{equation}\label{eq:DefinitionOfVectorsti}
 t^i = \dfrac{1}{\xi_i} (\mu \lambda_1, \ldots, \mu \lambda_{i-1}, \lambda_{i}, \ldots, \lambda_{N}) \quad \forall i = 1, \ldots, N
\end{equation}
where $\xi_i > 0$ is a normalising constant such that $\| t^i \|_1 = 1$. Define the coefficients $k_i$ through
\begin{equation}\label{defnk}
 k_i = \dfrac{\mu \lambda_1 + \ldots + \mu \lambda_{i} +  \lambda_{i+1} + \ldots +  \lambda_{N}}{(\mu - 1) \lambda_i} \quad i = 1, \ldots, N.
\end{equation}
Then define the matrices $D_j$ with $j = 1, \ldots, N-1$ and $Q$ by solving recursively
\begin{equation}\label{eq:ComputingTheRankOneMatrices}
 \sum_{j = 1}^{N} t^i_j X_j = Q + D_1 + \ldots + D_{i-1}, \quad \forall i = 1, \ldots, N
\end{equation}
and setting $D_N := - D_1 - \ldots - D_{N-1}$.
Then $P, D_1, \ldots, D_N, k_1, \ldots, k_N, X_1, \ldots, X_N$ solve equations \eqref{eq:TNEquations}.
\end{corollary}

\begin{remark}\label{rmk:DeterminantFormulaWithTi}
With the vectors $t^i$ defined as in \eqref{eq:DefinitionOfVectorsti}, Proposition~\ref{prop:CharacterizationTN} implies that
\begin{equation}\label{eq:DeterminantFormula}
 \sum_{j = 1}^N t_j^i \det (X_j^Z - X_i^Z) = 0, \quad \forall i \in \{1, \ldots, N\}, \, \forall Z \in \mathcal{A}_2.
\end{equation}
\end{remark}


\section{Inclusion sets in $K_a$}\label{PropKa}
In this section we start our study of subsets $\{ X_1, \ldots, X_N \} \subset \R^{3 \times 2}$ of $K_a$, for $a$ a strictly convex and increasing function with $a(0) = 0$ as in Theorem \ref{teo1}. These properties will then be used to prove our main Theorems \ref{teo1} and \ref{teo2} in the next section. For future reference, we define $Z_{12}, Z_{13} \in \mathcal{A}_2$ as
\[
 Z_{12} = ((1,2), (1,2)) \in \mathcal{A}_2 \quad \text{and} \quad Z_{13} = ((1,3), (1,2)) \in \mathcal{A}_2.
\]

\subsection{General inclusion sets}
We begin by studying general inclusion sets and find that a number of inequalities must be fulfilled in order to have $\{ X_1, \ldots, X_N\} \subset K_a$.
\begin{proposition}
If a set $\{ X_1, \ldots, X_N\} \subset \R^{3 \times 2}$ is contained in $K_a$ and
\[
 X_i =
 \begin{bmatrix}
 a(v_i) & u_i \\
 u_i & v_i \\
 u_i a(v_i) & \frac{1}{2} u_i^2 + F(v_i)
 \end{bmatrix}
 \quad
 \forall i = 1, \ldots, N,
\]
then for all $i,j \in \{ 1, \ldots, N \}$
\begin{equation}\label{eq:ConvexityIncreasingEquation}
 (F(v_j) - F(v_i))(a(v_j) - a(v_i)) < \dfrac{(a(v_j) + a(v_i))(v_j - v_i )(a(v_j) - a(v_i))}{2}
\end{equation}
whenever $v_j \neq v_i$.
\end{proposition}
\begin{proof}
Let $\{ X_1, \ldots, X_N \} \subset K_a$. Let $v_i \neq v_j$. Then, since $a$ is monotone increasing and strictly convex,
\begin{align*}
(F(v_j) - F(v_i))(a(v_j)-a(v_i)) &= (a(v_j)-a(v_i))\int_{v_i}^{v_j} a(t) \, dt \\
&= (a(v_j)-a(v_i))(v_j - v_i) \int_0^1 a(v_i + s(v_j - v_i)) \, ds \\
&= (a(v_j)-a(v_i))(v_j - v_i) \int_0^1 a(s v_j + (1-s) v_i)) \, ds \\ 
&<(a(v_j)-a(v_i))(v_j - v_i) \int_0^1 [s a(v_j) + (1-s) a(v_i)] \, ds \\
& = \dfrac{(a(v_j) + a(v_i))(v_j - v_i)(a(v_j)-a(v_i))}{2}.
\end{align*}
\end{proof}

\subsection{Properties of $T_N$ configurations contained in $K_a$}
We now turn our attention to the particular case in which $\{ X_1, \ldots, X_N \}$ induces a $T_N$ configuration.

\begin{proposition}\label{prop:NotAllTheSame}
If a set $\{ X_1, \ldots, X_N \} \subset K_a$ induces a $T_N$ configuration with
\[
 X_i =
 \begin{bmatrix}
 a(v_i) & u_i \\
 u_i & v_i \\
 u_i a(v_i) & \frac{1}{2} u_i^2 + F(v_i)
 \end{bmatrix}
 \quad
 \forall i = 1, \ldots, N.
\]
Then there must exist $i \neq j$ such that $v_i \neq v_j$.
\end{proposition}
\begin{proof}
For a contradiction, assume that all $v_i$ take the same value. Remark~\ref{rmk:DeterminantFormulaWithTi} yields
\begin{align*}
 0 = \sum_{j = 1}^{N} t^i_j \det (X_j^{Z_{12}} - X_i^{Z_{12}}) &= \sum_{j = 1}^{N} t^i_j [\underbrace{(a(v_j) - a(v_i))(v_j - v_i)}_{= 0} - (u_j - u_i)^2] \\
 &= - \sum_{j = 1}^{N} t^i_j (u_j - u_i)^2.
\end{align*}
This means that all $u_i$ take the same value and therefore all $X_i$ are identical. This contradicts our definition of $T_N$ configurations.
\end{proof}
Before going further, we introduce the following notation. If $\{ X_1, \ldots, X_N \}$ induces a $T_N$ configuration with
\begin{equation}\label{formX}
X_i = Q + D_1 + \dots + D_{i-1} + k_iD_i, \quad \rank{D_i} \le 1, \quad \sum_\ell D_\ell = 0,
\end{equation}
 then we define  $\{ Y_1, \ldots, Y_N \}$ as 
\begin{equation}\label{eq:RestrictionToTheUpperSquareMatrix}
 Y_j = X_j^{Z_{12}}, C_j = D_j^{Z_{12}} \quad \forall j = 1, \ldots, N,\quad \text{ and } P = Q^{Z_{12}}.
\end{equation}
We will write
\begin{equation}\label{PL}
P_0 = P_1 = P, P_\ell := P + C_1 + \dots + C_{\ell - 1}, \quad\forall \ell \ge 1.
\end{equation}
In other words, this operation simply consists in restricting our attention to the first two lines of the $3\times 2$ matrices $X_i$. Notice that, by \eqref{eq:ComputingTheRankOneMatrices}, it follows that $\{ C_1, \ldots, C_N \} \subset \R^{2 \times 2}_{\text{sym}}$. We use the above notation throughout the remainder of this paper.

\begin{lemma}\label{lemma:FirstLemmaStarQuantity}
Let $\{ X_1, \ldots, X_N \} \subset K_a$ induce a $T_N$ configuration of the form \eqref{formX}. Then, in the notation above and for $\xi_i$ as in \eqref{eq:DefinitionOfVectorsti}, we have
 \begin{align*}
  (\star)_i &:= \xi_i \sum_{j = 1}^{N} t^i_j (X_j)_{11} \det(X_j^{Z_{12}} - X_i^{Z_{12}}) = \xi_i \sum_{j = 1}^{N} t^i_j (Y_j)_{11} \det(Y_j - Y_i) > 0 \quad \forall  i = 1, \ldots, N.
 \end{align*}
\end{lemma}
Let us point out that the reason for which we define the quantity $(\star)_i$ with a factor $\xi_i > 0$ is purely practical. Indeed, it is only useful to obtain a cleaner expression in the forthcoming \eqref{eq:FinalExpressionOfTheStarQuantity}.
\begin{proof}
By Remark~\ref{rmk:DeterminantFormulaWithTi} we have 
\begin{align*}
 0 &= \sum_{j = 1}^N t_j^i \det (X_j^{Z_{13}} - X_i^{Z_{13}}) \\
 &= \sum_{j = 1}^N t_j^i \left[(a(v_j) - a(v_i))\left(\frac{1}{2} u_j^2 - \frac{1}{2} u_i^2 + F(v_j) - F(v_i)\right) - (u_j - u_i)(u_j a(v_j) - u_i a(v_i))\right] \\
 &= \sum_{j = 1}^N t_j^i (a(v_j) - a(v_i))\left(F(v_j) - F(v_i)\right) \\
 & \qquad + \sum_{j = 1}^N t_j^i \underbrace{\left[(a(v_j) - a(v_i))\left(\frac{1}{2} u_j^2 - \frac{1}{2} u_i^2 \right) - (u_j - u_i)(u_j a(v_j) - u_{i} a(v_i))\right]}_{\qquad =: \Phi_{i,j}} \\
\end{align*}
Firstly, by \eqref{eq:ConvexityIncreasingEquation} combined with the fact that all $v_j$ are not identical, see Proposition~\ref{prop:NotAllTheSame}, we deduce
\begin{equation}\label{eq:StarQuantityLemmaConvexityIncreasingPart}
 \sum_{j = 1}^N t_j^i (a(v_j) - a(v_i))\left(F(v_j) - F(v_i)\right) < \sum_{j = 1}^N t_j^i \dfrac{(a(v_j) + a(v_i))(v_j - v_i )(a(v_j) - a(v_i))}{2}.
\end{equation}
Secondly, we compute the quantities $\Phi_{i,j}$ for all $i$ and $j$.
We have, by direct computation,
\begin{align}
\begin{split} \label{eq:ComputationsQuantityPhi}
\Phi_{i,j} &= - \dfrac{1}{2}(u_j - u_i)^2 (a(v_j) + a(v_i)).
\end{split}
\end{align}
Then by combining \eqref{eq:StarQuantityLemmaConvexityIncreasingPart} and \eqref{eq:ComputationsQuantityPhi}, we find
\begin{align*}
 0 &= \sum_{j = 1}^N t_j^i \det (X_j^{Z_{13}} - X_i^{Z_{13}}) \\
 &< \dfrac{1}{2} \sum_{j = 1}^N t_j^i (a(v_j) + a(v_i))(v_j - v_i )(a(v_j) - a(v_i)) - \dfrac{1}{2} \sum_{j = 1}^N t_j^i (u_j - u_i)^2 (a(v_j) + a(v_i)) \\
 &= \dfrac{1}{2} \sum_{j = 1}^N t_j^i (a(v_j) + a(v_i))( (v_j - v_i )(a(v_j) - a(v_i)) - (u_j - u_i)^2 ) \\
 &= \dfrac{1}{2} \sum_{j = 1}^N t_j^i (a(v_j) + a(v_i)) \det (X_j^{Z_{12}} - X_i^{Z_{12}}). 
\end{align*}
Since by Remark~\ref{rmk:DeterminantFormulaWithTi} we have
\[
 \sum_{j = 1}^N t_j^i \det (X_j^{Z_{12}} - X_i^{Z_{12}}) = 0,
\]
we deduce that
\[
 \sum_{j = 1}^N t_j^i a(v_j) \det (X_j^{Z_{12}} - X_i^{Z_{12}}) > 0.
\]
Equivalently, the latter can be rewritten as
\[
 0 < \sum_{j = 1}^N t_j^i (X_j)_{11} \det (X_j^{Z_{12}} - X_i^{Z_{12}}) = \sum_{j = 1}^N t_j^i (Y_j)_{11} \det (Y_j - Y_i),
\]
which concludes the proof.
\end{proof}

The next proposition, whose proof is postponed to the end of the section, gives us a new formulation of $(\star)_i$.

\begin{lemma}\label{eq:IntermediaryExpressionOfStar}
With the notation introduced in \eqref{eq:RestrictionToTheUpperSquareMatrix} and \eqref{PL}, we have
\[
 (\star)_i = \xi_i \sum_{\alpha = 1}^N k_{\alpha} (k_{\alpha} - 1) t^i_{\alpha} (C_{\alpha})_{11} \langle \cofactor C_{\alpha}, P_{\alpha} - Y_i \rangle \quad \forall i = 1, \ldots, N.
\]
\end{lemma}

Before stating the final result of this section, we introduce some more notation. Let $\{ X_1, \ldots, X_N \}$ be a $T_N$ configuration. As above, let $\{ Y_1 , \ldots, Y_N \}$ be the $T_N$ configuration defined by \eqref{eq:RestrictionToTheUpperSquareMatrix} and $\{C_1, \ldots, C_N\}$ the rank-one matrices of this $T_N$ configuration. Recall also the definitions of $k_i$ of Definition \ref{deftn} and of $\lambda_i$ and $\mu$ of Proposition \ref{prop:CharacterizationTN}. We define $S_{\beta}$ and $T_{\beta}$ for all $\beta = 1, \ldots, N$ as 
\[
S_1:= 0 \text{ and } S_{\beta} = \sum_{\alpha = 1}^{\beta - 1} k_{\alpha} (k_{\alpha} - 1) \lambda_\alpha (C_{\alpha})_{11} \cofactor C_{\alpha}, \quad \forall 2 \le \beta \le N
\]
and
\[
T_N:= 0 \text{ and } T_{\beta} = \sum_{\alpha = \beta+1}^{N} k_{\alpha} (k_{\alpha} - 1) \lambda_\alpha (C_{\alpha})_{11} \cofactor C_{\alpha}, \quad \forall 1 \le \beta \le N-1.
\]
We can now prove the following proposition, which is the main result of this section.
\begin{proposition}
Let $\{ X_1, \ldots, X_N \} \subset \R^{3 \times 2}$. In the notation above, we have
\begin{equation}\label{eq:FinalExpressionOfTheStarQuantity}
 (\star)_i = - \mu \sum_{\beta = 1}^{i-1} \langle S_\beta, C_\beta \rangle + \sum_{\beta = i+1}^{N} \langle T_\beta, C_\beta \rangle - \mu k_i \langle S_i, C_i \rangle + (1 - k_i) \langle T_i, C_i \rangle.
\end{equation}
\end{proposition}
\begin{proof}
By Lemma~\ref{eq:IntermediaryExpressionOfStar}, 
\[
 (\star)_i = \xi_i \sum_{\alpha = 1}^N k_{\alpha} (k_{\alpha} - 1) t^i_{\alpha} (C_{\alpha})_{11} \langle \cofactor C_{\alpha}, P_{\alpha} - Y_i \rangle \quad \forall i = 1, \ldots, N.
\]
By computing the quantities $P_{\alpha} - Y_i$ in terms of rank-one matrices $C_1, \ldots, C_N$ and exploiting the fact that, for all $1 \le \alpha \le N$,
\[
\langle \cofactor C_\alpha, C_\alpha\rangle = 2\det(C_\alpha) = 0,
\]
we find
\begin{align*}
 (\star)_i &= - \mu \sum_{\alpha = 1}^{i-1} \sum_{\beta = \alpha + 1}^{i-1} k_{\alpha} (k_{\alpha} - 1) \lambda_{\alpha} (C_{\alpha})_{11} \langle \cofactor C_{\alpha}, C_{\beta} \rangle \\
 & \qquad - \mu k_i \sum_{\alpha = 1}^{i-1} k_{\alpha} (k_{\alpha} - 1) \lambda_{\alpha} (C_{\alpha})_{11} \langle \cofactor C_{\alpha}, C_i \rangle \\
 & \qquad + (1 - k_i) \sum_{\alpha = i+1}^{N} k_{\alpha} (k_{\alpha} - 1) \lambda_{\alpha} (C_{\alpha})_{11} \langle \cofactor C_{\alpha}, C_i \rangle \\
 & \qquad + \sum_{\alpha = i+1}^{N} \sum_{\beta = i + 1}^{\alpha - 1} k_{\alpha} (k_{\alpha} - 1) \lambda_{\alpha} (C_{\alpha})_{11} \langle \cofactor C_{\alpha}, C_{\beta} \rangle.
\end{align*}
By permuting the sums, the quantities $S_\beta$ and $T_\beta$ appear and we find \eqref{eq:FinalExpressionOfTheStarQuantity}.
\end{proof}

\begin{remark}
We make some observations about the quantities $S_{\beta}$ and $T_{\beta}$. We see that all the matrices in the sums defining $S_{\beta}$ and $T_{\beta}$ are symmetric positive semi-definite. Moreover, by the definitions of $S_{\beta}$ and $T_{\beta}$ we have
\begin{equation}
 S_{\beta} \leq S_{\beta+1} \quad \text{and} \quad T_{\beta} \geq T_{\beta+1}, \quad \forall \beta \in \{ 1, \ldots, N-1 \}.
\end{equation}
Finally, we also notice that $S_1 = 0$, $T_N = 0$ and since $\langle \cofactor C_j , C_j \rangle = 2\det(C_j) = 0$ for all $j$, we deduce that
\begin{equation}
 \langle S_{\beta} , C_{\beta} \rangle = \langle S_{\beta+1} , C_{\beta} \rangle \quad \text{and} \quad \langle T_{\beta} , C_{\beta} \rangle = \langle T_{\beta - 1} , C_{\beta} \rangle, \quad \forall \beta \in \{ 1, \ldots, N-1 \}.
\end{equation}
\end{remark}
We finish this section by proving Lemma~\ref{eq:IntermediaryExpressionOfStar}. In the proof we will need the following technical result.

\begin{lemma}\label{l:linear}
Assume the real numbers $\mu>1$, $\lambda_1, \ldots , \lambda_N >0$ and $k_1, \ldots , k_N >1$ are linked by the formulas \eqref{defnk}. Assume $p, x_j, c_j$ are numbers satisfying the  relations
\begin{align*}
x_i &= p + c_1 + \ldots + c_{i-1} + k_i c_i\\
0 &= c_1+ \ldots + c_N\, .
\end{align*}
If we define the vectors $t^i$ as in \eqref{eq:DefinitionOfVectorsti}, then, for any $1 \le i \le N$,
\begin{equation}\label{e:sumfinale}
\sum_j t^i_j x_j = p + c_1 + \ldots + c_{i-1}\, .
\end{equation}
Furthermore,  if $q,y_j,d_j$ are numbers again fulfilling, for all $1\le i \le N$,
\begin{align*}
y_i &= q + d_1 + \ldots + d_{i-1} + k_i d_i\\
0 &= d_1+ \ldots + d_N\,,
\end{align*}
then
\begin{equation}\label{e:sumfinalquad}
\sum_j t^i_j x_jy_j = \sum_{j}k_j(k_j-1)t^i_jc_jd_j + \sum_{j,k}^{i-1}c_jd_k  + pq + p\sum_{j = 1}^{i-1}d_j + q\sum_{j = 1}^{i-1}c_j
\end{equation}
\end{lemma}
\begin{proof}
We refer the reader to \cite[Lemma 3.10]{GMTDI} for a proof of \eqref{e:sumfinale}. Let us then fix $i \in \{1,\dots, N\}$ and show \eqref{e:sumfinalquad}. By writing
\[
\sum_j t^i_j x_jy_j  = \sum_j t^i_j (x_j-p)y_j + p\left(\sum_jt_j^iy_j\right) = \sum_j t^i_j (x_j-p)(y_j-q) + p\left(\sum_jt_j^iy_j\right) + q\sum_{j}t^i_j(x_j-p),
\]
we see that the last two terms can be computed using \eqref{e:sumfinale}. Thus, we add the assumption $p = q = 0$ and compute \eqref{e:sumfinale} in this simplified setting. The next computation is entirely analogous to the one of \cite[Theorem 1.2]{GMTDI}, and we repeat it for the reader's convenience. We start by computing the sum for $i = 1$, $\sum_j\lambda_jx_jy_j.$ We rewrite it in the following way:

\begin{equation}\label{quadsum}
\begin{split}
 \sum_j\lambda_j x_jy_j =& \sum_{j = 1}^N\lambda_j\left(\sum_{1\le a,b\le j - 1}c_ad_b + k_j\sum_{1\le a \le j - 1} c_ad_j + k_j\sum_{1\le b \le j - 1} c_jd_b + k_j^2c_jd_j\right) = \sum_{i,j}g_{ij}c_id_j,
\end{split}
\end{equation}
where
\[
g_{ij}=
\begin{cases}
 \lambda_ik_i + \sum_{r = i + 1}^N\lambda_r,\text{ if } i > j,\\
 \lambda_jk_j + \sum_{r = j + 1}^N\lambda_r,\text{ if } i < j,\\
\lambda_ik_i^2 + \sum_{r = i + 1}^N\lambda_r,\text{ if } i = j.
\end{cases}
\]
Using \eqref{defnk}, we compute:
\[
g_{ij} = g_{ji} = \lambda_ik_i + \sum_{r = i + 1}^N\lambda_r = \frac{\mu}{\mu - 1},\quad g_{ii} = k_i^2\lambda_i + \sum_{r = i + 1}^N\lambda_r = k_i(k_i - 1)\lambda_i + \frac{\mu}{\mu - 1}.
\]
Since $\sum_\ell c_\ell = 0 = \sum_\ell d_\ell$, then also $\sum_{i,j}c_id_j = 0$, and so $\sum_{i\neq j}c_id_j = -\sum_ic_id_i$. Hence, \eqref{quadsum} becomes
\[
 \sum_{i,j}g_{ij}c_id_j = \frac{\mu}{\mu - 1}\sum_{i\neq j}c_id_j + \sum_i\left( k_i(k_i - 1)\lambda_i + \frac{\mu}{\mu - 1}\right)c_id_i =  \sum_i k_i(k_i - 1)\lambda_ic_id_i.
\]
We just proved that
\begin{equation}\label{b}
\sum_j\lambda_jx_jy_j =  \sum_i k_i(k_i - 1)\lambda_ic_id_i,
\end{equation}
that is \eqref{e:sumfinalquad} in the case $i = 1$ and $p = q = 0$. Now we show it for $i$. Recall that 
\[
t^i= \frac{1}{\xi_i}(\mu\lambda_1,\dots,\mu\lambda_{i - 1},\lambda_i,\dots,\lambda_N)\, 
\]
and hence
\[
\sum_j t_j^ix_jy_j = \frac{1}{\xi_i}\left[(\mu-1)\sum_{j = 1}^{i-1}\lambda_jx_jy_j + \sum_{j = 1}^{N}\lambda_jx_jy_j \right] = \frac{1}{\xi_i}\left[(\mu-1)\sum_{j = 1}^{i-1}\lambda_jx_jy_j +\sum_j k_j(k_j - 1)\lambda_jc_jd_j \right]
\]
We again express the sum up to $i - 1$ in the following way:
\[
\sum_{j = 1}^{i - 1}\lambda_jx_jy_j = \sum_{k,j}^{i - 1}s_{kj}c_kd_j,
\]
where
\[
s_{\alpha\beta} =
\begin{cases}
 k_\alpha\lambda_\alpha + \dots + \lambda_{i - 1}, \text{ if }\alpha>\beta\\
 k_\beta\lambda_\beta + \dots + \lambda_{i - 1}, \text{ if }\alpha<\beta\\
 k_\alpha^2\lambda_\alpha + \dots + \lambda_{i - 1}, \text{ if } \alpha = \beta.
 \end{cases}
\]
Now
\[
 k_r\lambda_r + \dots + \lambda_{i - 1}=\frac{\mu( \sum_{j = 1}^{i - 1}\lambda_j) + \sum_{j = i}^{N}\lambda_j }{\mu - 1}
\]
and so
\[
 k_r\lambda_r + \dots + \lambda_{i - 1}=\frac{(\mu -1)( \sum_{j = 1}^{i - 1}\lambda_j) + 1 }{\mu - 1} = \frac{\xi_i}{\mu - 1} =: b_{i - 1}
\]
Hence
\[
\sum_{j = 1}^{i - 1}\lambda_j x_jy_j = \sum_{k,j}^{i - 1}s_{kj}c_kd_j = b_{i - 1}\sum_{k,j}^{i - 1}c_kd_j + \sum_{\alpha = 1}^{i - 1} k_\alpha(k_\alpha - 1)\lambda_\alpha c_\alpha d_\alpha
\]
and we conclude
\[
\sum_j t_j^ix_jy_j = \frac{1}{\xi_i}\left[(\mu-1)\sum_{j = 1}^{i-1}\lambda_jx_jy_j +\sum_j k_j(k_j - 1)\lambda_jc_jd_j \right] = \sum_{j}k_j(k_j-1)t^i_jc_jd_j
+ \sum_{j,k}^{i-1}c_jd_k.
\]
\end{proof}

\begin{proof}[Proof of Lemma~\ref{eq:IntermediaryExpressionOfStar}]

Fix $i \in \{1,\dots, N\}.$ Notice that, for any matrix $A 	\in \R^{2\times 2}$,
\[
(\star)_i = \xi_i \sum_{j = 1}^{N} t^i_j (Y_j)_{11} \det(Y_j - Y_i) = (\star)_i = \xi_i \sum_{j = 1}^{N} t^i_j (Y_j- A)_{11} \det(Y_j-A - (Y_i- A)),
\]
since by Proposition \ref{prop:CharacterizationTN},
\[
\sum_{j = 1}^{N} t^i_j \det(Y_j-Y_i) = 0.
\]
Thus, even though $P$ of \eqref{eq:RestrictionToTheUpperSquareMatrix} may not, in principle, be $0$, we can shift $\{Y_1,\dots, Y_N\}$ by $-P$ to add the additional assumption that
\begin{equation}\label{newYi}
Y_i = C_1 + \dots + C_{i-1} + k_iC_i, \text{ for }C_\ell \in \R^{2\times 2}_{\text{sym}},\; \det(C_\ell) = 0,\;\sum_\ell^N C_\ell = 0,
\end{equation}
without changing the value of $(\star)_i$. Since 
\begin{equation}\label{detAB}
\det(A + B) = \det(A) + \langle(\cofactor A)^T, B \rangle + \det(B)
\end{equation}
 for all $A,B \in \R^{2\times 2}$ and $(\cofactor A)^T = \cofactor A$ for all $A \in \R^{2 \times 2}_{\text{sym}}$, we can write
\begin{align*}
(\star)_i &= \xi_i \sum_{j = 1}^{N} t^i_j (Y_j)_{11} \det(Y_j - Y_i) \\
&= \xi_i \underbrace{\sum_{j = 1}^{N} t^i_j (Y_j)_{11} \det(Y_j)}_{=: I} -\xi_i \underbrace{\sum_{j = 1}^{N} t^i_j (Y_j)_{11} \langle\cofactor(Y_j),Y_i\rangle}_{=:II} +  \xi_i\underbrace{\det(Y_i) \sum_{j = 1}^{N} t^i_j (Y_j)_{11}}_{=: III}.
\end{align*}
We start by computing $I$. Recalling \eqref{PL}, let
\[
d_\ell := \langle (\cofactor C_\ell)^T, P_\ell\rangle = \langle \cofactor C_\ell, P_\ell\rangle .
\]
For $\{Y_1,\dots, Y_N\}$ given by \eqref{newYi}, we have
\begin{equation}\label{TN:det}
\det(Y_\ell) =  d_1 + \dots + k_\ell d_\ell,
\end{equation}
which can be shown inductively using \eqref{newYi} and \eqref{detAB}. Moreover, again using \eqref{detAB}, we rewrite
\[
d_\ell = \det(P_{\ell + 1}) -\det(P_\ell).
\]
It follows that
\begin{equation}\label{sum:zero}
\sum_{\ell = 1}^Nd_\ell = \det(P_{N + 1}) - \det(P_0) = 0.
\end{equation}
Equations \eqref{newYi}-\eqref{TN:det}-\eqref{sum:zero} allow us to use \eqref{e:sumfinalquad} of Lemma \ref{l:linear} to compute
\begin{equation}\label{I}
I =\sum_{j = 1}^{N} t^i_j (Y_j)_{11} \det(Y_j) = \sum_{j}k_j(k_j-1)t^i_j(C_j)_{11}\langle \cofactor C_j, P_j\rangle + \sum_{j,k}^{i-1}(C_j)_{11}\langle \cofactor C_k, P_k\rangle.
\end{equation}
Again \eqref{e:sumfinalquad} of Lemma \ref{l:linear} (applied componentwise) implies
\begin{equation}\label{II}
II = \sum_{j = 1}^{N} t^i_j (Y_j)_{11} \langle\cofactor(Y_j),Y_i\rangle = \sum_{j}k_j(k_j-1)t^i_j(C_j)_{11}\langle \cofactor C_j, Y_i\rangle+ \sum_{j,k}^{i-1}(C_j)_{11}\langle \cofactor C_k, Y_i\rangle
\end{equation}
Finally, \eqref{e:sumfinale} yields
\begin{equation}\label{III}
III =\det(Y_i) \sum_{j = 1}^{N} t^i_j (Y_j)_{11} = \det(Y_i)\sum_{j = 1}^{i-1}(C_{j})_{11}.
\end{equation}
Thus, combining \eqref{I}-\eqref{II}-\eqref{III}, we obtain
\begin{align*}
I- II + III &=  \sum_{j}k_j(k_j-1)t^i_j(C_j)_{11}\langle \cofactor C_j, P_j - Y_i\rangle + \sum_{j,k}^{i-1}(C_j)_{11}\langle \cofactor C_k, P_k\rangle \\
&- \sum_{j,k}^{i-1}(C_j)_{11}\langle \cofactor C_k, Y_i\rangle + \det(Y_i)\sum_{j = 1}^{i-1}(C_{j})_{11}.
\end{align*}
The proof is thus concluded if we show
\begin{equation}\label{sumzero:fine}
\sum_{j,k}^{i-1}(C_j)_{11}\langle \cofactor C_k, P_k\rangle - \sum_{j,k}^{i-1}(C_j)_{11}\langle \cofactor C_k, Y_i\rangle + \det(Y_i)\sum_{j = 1}^{i-1}(C_{j})_{11} = 0.
\end{equation}
To see this, we compute
\begin{align*}
\sum_{k = 1}^{i-1}\langle \cofactor C_k, P_k - Y_i\rangle &\overset{\eqref{detAB}}{=} \sum_{k = 1}^{i-1} [\det(P_{k + 1}- Y_i)- \det(P_k - Y_i)] \\
&= \det(P_i - Y_i) - \det(P_1 - Y_i) = \det(k_iC_i) - \det(Y_i) = -\det(Y_i)
\end{align*}
This shows \eqref{sumzero:fine} and concludes the proof.
\end{proof}

\section{Proof of Theorem \ref{teo1} and \ref{teo2}}\label{TheEnd}
In this section, we prove Theorem \ref{teo1} and \ref{teo2}. Both proofs will be split into different cases depending on the sign of the rank-one matrices $C_j$. By considering all possible cases, we will see that there always is an index $i$ such that the $(\star)_i \leq 0$, thus contradicting Lemma~\ref{lemma:FirstLemmaStarQuantity} and hence excluding the existence of a $T_N$ configuration contained in $K_a$. 
Throughout this section, we assume that $\{ X_1, \ldots, X_N \}$ is a $T_N$ configuration. We let $\{ Y_1, \ldots, Y_N \}$ be the $T_N$ configuration given by \eqref{eq:RestrictionToTheUpperSquareMatrix} and $\{ C_1, \ldots, C_N \}$ the rank-one matrices of this $T_N$ configuration.
Since $ C_j \in \R^{2 \times 2}_{\text{sym}}$ for all $1\le j \le N$ is a rank-one matrix, $C_j$ is either positive semi-definite or negative semi-definite. As usual for symmetric matrices, we write $C_j \leq 0$ when $C_j$ is negative semi-definite and $C_j \geq 0$ when $C_j$ is positive semi-definite. In particular, we will say that $C_j$ has $-$ sign when $C_j \leq 0$ and has $+$ sign when $C_j \geq 0$. Let us introduce the following notation to denote the cases that we consider. Let $s_1, \ldots, s_N$ be signs, i.e. $s_i = +$ or $-$. When saying that we consider the case
\[
 (s_1 \ldots s_N),
\]
we mean that we are reducing ourselves to the case where each $C_j$ has sign $s_j$.
For example for $N=4$, $(+ + + -)$ denotes the case where $C_1, C_2, C_3 \geq 0$ and $C_4 \leq 0$.
Before reducing ourselves to the cases $N = 4$ and $N = 5$, we can exclude some cases for general $N$. 

\begin{proposition}\label{prop:SuperEasyCase}
All the rank-one matrices $C_j$ cannot have the same sign.
\end{proposition}

\begin{proof}
This follows immediately from the fact that $\sum_{i = 1}^N C_i = 0$ and our definition of $T_N$ configurations. 
\end{proof}

\begin{proposition}\label{prop:EasyCaseOne}
If there is $j_0$ such that $C_{j_0}$ has different sign than all the other rank-one matrices $C_j$, $j \neq j_0$, then there is $i$ such that $(\star)_i \leq 0$.
\end{proposition}

\begin{proof}
The proof is divided into two cases.

\fbox{Case 1: ( $\exists j_0$ such that $C_{j_0} \geq 0$ and $C_j \leq 0, \forall j \neq j_0$)}
We have
\begin{align*}
(\star)_{j_0} &= - \mu (\langle S_1, C_1 \rangle + \ldots + \langle S_{j_0-1}, C_{j_0-1} \rangle + k_{j_0} \langle S_{j_0}, C_{j_0} \rangle) \\
& \qquad + (1 - k_{j_0}) \langle T_{j_0}, C_{j_0} \rangle + \langle T_{j_0+1}, C_{j_0+1} \rangle + \ldots + \langle T_{N}, C_{N} \rangle \\
&\leq - \mu (\langle S_1, C_1 \rangle + \ldots + \langle S_{j_0-1}, C_{j_0-1} \rangle + k_{j_0} \langle S_{j_0}, C_{j_0} \rangle) \\
&\leq - \mu (\langle S_1, C_1 \rangle + \ldots + \langle S_{j_0-1}, C_{j_0-1} \rangle + \langle S_{j_0}, C_{j_0} \rangle) \\
&\leq - \mu \langle S_{j_0}, C_1 + \ldots + C_{j_0} \rangle \\
&= \mu \langle S_{j_0}, C_{j_0+1} + \ldots + C_N \rangle \leq 0.\\
\end{align*}

\fbox{Case 2: ($\exists j_0$ such that $C_{j_0} \leq 0$ and $C_j \geq 0, \forall j \neq j_0$)}
In this case, if $j_0 > 1$
\begin{align*}
(\star)_{j_0-1} &= - \mu (\langle S_1, C_1 \rangle + \ldots + \langle S_{j_0-2}, C_{j_0-2} \rangle + k_{j_0-1} \langle S_{j_0-1}, C_{j_0-1} \rangle) \\
& \qquad + (1 - k_{j_0-1}) \langle T_{j_0-1}, C_{j_0-1} \rangle + \langle T_{j_0}, C_{j_0} \rangle + \ldots + \langle T_{N}, C_{N} \rangle \\
&\leq \langle T_{j_0}, C_{j_0} \rangle + \ldots + \langle T_{N}, C_{N} \rangle \\
&\leq \langle T_{j_0}, C_{j_0} \rangle + \langle T_{j_0}, C_{j_0+1} \rangle + \ldots + \langle T_{j_0}, C_{N} \rangle \\
&= \langle T_{j_0}, C_{j_0} + \ldots + C_{N} \rangle \\
&= - \langle T_{j_0}, C_1 + \ldots + C_{j_0-1} \rangle \leq 0.
\end{align*}
On the other hand, if $j_0 = 1$, the proof follows immediately considering $(\star)_{N}$ and using that $\langle S_1, C_1\rangle = 0$.
\end{proof}

\begin{proposition}\label{prop:EasyCaseTwo}
If there is $j_0 \in \{ 2, \ldots, N-1 \}$ such that $C_j \geq 0$ for all $2 \leq j \leq j_0$ and $C_j \leq 0$ for all $j_0 + 1 \leq j \leq N-1$, then there exists $i$ such that $(\star)_i \leq 0$.
\end{proposition}

\begin{proof}
We have
\[
 (\star)_{j_0} = - \mu \sum_{\beta = 1}^{j_0-1} \underbrace{\langle S_\beta, C_\beta \rangle}_{\geq 0} + \sum_{\beta = j_0+1}^{N} \underbrace{\langle T_\beta, C_\beta \rangle}_{\leq 0} - \mu k_{j_0} \underbrace{\langle S_{j_0}, C_{j_0} \rangle}_{\geq 0} + (1 - k_{j_0}) \underbrace{\langle T_{j_0}, C_{j_0} \rangle}_{\geq 0} \leq 0
\]
as desired. This finishes the proof.
\end{proof}

We can now prove Theorem \ref{teo1}.

\begin{proof}[Proof of Theorem \ref{teo1}]
Let $\{ X_1, \ldots, X_4 \}$ be a $T_4$ configuration and assume by contradiction that $$\{ X_1, \ldots, X_4 \} \subset K_a$$ for some $a$ which is convex and increasing. Then let $\{ Y_1, \ldots, Y_4 \}$ be the $T_4$ configuration given by \eqref{eq:RestrictionToTheUpperSquareMatrix}. We will prove that there exists an $i$ such that $(\star)_i \leq 0$. To this end, we consider all possible cases of combinations of signs of the matrices $C_j$. In the case $N = 4$ that we are considering there are 16 cases. The cases $(+ + + +)$ and $(- - - -)$ can already be excluded by Proposition~\ref{prop:SuperEasyCase}. Due to Proposition~\ref{prop:EasyCaseOne} and Proposition~\ref{prop:EasyCaseTwo}, we can exclude 11 more cases and the only cases that remain are
\begin{enumerate}
\item $(+ - + -),$ \label{item:N4Case1}
\medskip
\item $(+ - - +)$ and \label{item:N4Case2}
\medskip
\item $(- - + +).$ \label{item:N4Case3}
\end{enumerate}
We now prove that in all of these three cases there is an index $i$ such that $(\star)_i \leq 0$.

\emph{Case \ref{item:N4Case1}.}
Since $\langle S_1, C_1 \rangle = 0$ and $\langle S_3, C_3 \rangle, \langle T_3, C_3 \rangle \geq 0$, we have
\begin{align*}
(\star)_3 &= - \mu \langle S_2, C_2 \rangle - \mu k_3 \langle S_3, C_3 \rangle + (1 - k_3) \langle T_3, C_3 \rangle \leq - \mu (\langle S_2, C_2 \rangle + \langle S_3, C_3 \rangle) \\
&= - \mu (\langle S_2, C_1 \rangle + \langle S_2, C_2 \rangle + \langle S_2, C_3 \rangle) \leq - \mu \langle S_2, C_1 + C_2 + C_3 \rangle = \mu \langle S_2, C_4 \rangle \leq 0.
\end{align*}

\emph{Case \ref{item:N4Case2}.}
We have 
\[
 (\star)_1 = (1 - k_1) \langle T_1, C_1 \rangle + \langle T_2, C_2 \rangle + \langle T_3, C_3 \rangle \leq 0.
\]

\emph{Case \ref{item:N4Case3}.}
Since $\langle S_4, C_4 \rangle \geq 0$ and $\langle S_4, C_4 \rangle \geq \langle S_3, C_4 \rangle$
\begin{align*}
(\star)_4 &= - \mu \langle S_2, C_2 \rangle - \mu \langle S_3, C_3 \rangle - \mu k_4 \langle S_4, C_4 \rangle \leq - \mu (\langle S_2, C_2 \rangle + \langle S_3, C_3 \rangle + \langle S_4, C_4 \rangle) \\
&\leq - \mu (\langle S_3, C_2 \rangle + \langle S_3, C_3 \rangle + \langle S_3, C_4 \rangle) = - \mu \langle S_3, C_2 + C_3 + C_4 \rangle = \mu \langle S_3, C_1 \rangle \leq 0.
\end{align*}

Therefore we deduce that there is an index $i$ such that $(\star)_i \leq 0$. This finishes the proof.
\end{proof}

We now prove Theorem \ref{teo2}.

\begin{proof}[Proof of Theorem \ref{teo2}]
Let $\{ X_1, \ldots, X_5 \}$ be a $T_5$ configuration and assume by contradiction that $$\{ X_1, \ldots, X_5 \} \subset K_a$$ for some $a$ which is convex and increasing. Then let $\{ Y_1, \ldots, Y_5 \}$ be the $T_5$ configuration given by \eqref{eq:RestrictionToTheUpperSquareMatrix}. We will prove that there exists an $i$ such that $(\star)_i \leq 0$. To this end, we consider all possible cases of combinations of signs of the matrices $C_j$. In the case $N = 5$, there are 32 cases. By Proposition~\ref{prop:SuperEasyCase}, Proposition~\ref{prop:EasyCaseOne} and Proposition~\ref{prop:EasyCaseTwo}, we can exclude 19 cases. The 13 remaining cases are
\begin{enumerate}
\item (+ - + - -) \label{item:N5Case1}
\medskip
\item (+ - - + -) \label{item:N5Case2}
\medskip
\item (+ - - - +) \label{item:N5Case3}
\medskip
\item (- + - + -) \label{item:N5Case4}
\medskip
\item (- - + + -) \label{item:N5Case5}
\medskip
\item (- - + - +) \label{item:N5Case6}
\medskip
\item (- - - + +) \label{item:N5Case7}
\medskip
\item (- - + + +) \label{item:N5Case8}
\medskip
\item (- + - + +) \label{item:N5Case9}
\medskip
\item (+ - - + +) \label{item:N5Case10}
\medskip
\item (+ - + - +) \label{item:N5Case11}
\medskip
\item (+ - + + -) \label{item:N5Case12}
\medskip
\item (+ + - + -) \label{item:N5Case13}
\end{enumerate}

\emph{Case \ref{item:N5Case1}.} 
\begin{align*}
 (\star)_3 &= - \mu ( \langle S_2, C_2 \rangle + k_3 \langle S_3, C_3 \rangle ) + (1 - k_3) \langle T_3, C_3 \rangle + \langle T_4, C_4 \rangle \leq - \mu ( \langle S_2, C_2 \rangle + \langle S_3, C_3 \rangle ) \\
 &= - \mu ( \langle S_1, C_1 \rangle + \langle S_2, C_2 \rangle + \langle S_3, C_3 \rangle )\le - \mu \langle S_2, C_1 + C_2 + C_3 \rangle = \mu \langle S_2, C_4 + C_5 \rangle \leq 0
\end{align*}

\emph{Case \ref{item:N5Case2}.} 
\begin{align*}
 (\star)_1 &= - \mu k_1 \langle S_1, C_1 \rangle + (1 - k_1) \langle T_1, C_1 \rangle + \langle T_2, C_2 \rangle + \langle T_3, C_3 \rangle + \langle T_4, C_4 \rangle + \langle T_5, C_5 \rangle \\
 &\leq \langle T_2, C_2 \rangle + \langle T_3, C_3 \rangle + \langle T_4, C_4 \rangle + \langle T_5, C_5 \rangle \leq \langle T_4, C_2 + C_3 + C_4 + C_5 \rangle \\ 
 &= - \langle T_4, C_1 \rangle \leq 0
\end{align*}

\emph{Case \ref{item:N5Case3}.} 
\begin{align*}
 (\star)_1 &= (1 - k_1) \langle T_1, C_1 \rangle + \langle T_2, C_2 \rangle + \langle T_3, C_3 \rangle + \langle T_4, C_4 \rangle \leq 0
\end{align*}

\emph{Case \ref{item:N5Case4}.} 
\begin{align*}
 (\star)_4 &\leq - \mu (\langle S_2, C_2 \rangle + \langle S_3, C_3 \rangle + k_4 \langle S_4, C_4 \rangle) \leq - \mu (\langle S_3, C_2 \rangle + \langle S_3, C_3 \rangle + \langle S_3, C_4 \rangle) \\
 & = - \mu \langle S_3, C_2 + C_3 + C_4 \rangle = \mu \langle S_3, C_1 + C_5 \rangle \leq 0.
 \end{align*}
 
 \emph{Case \ref{item:N5Case5}.} 
\begin{align*}
 (\star)_4 &\leq - \mu (\langle S_2, C_2 \rangle + \langle S_3, C_3 \rangle + k_4 \langle S_4, C_4 \rangle) \leq - \mu (\langle S_3, C_2 \rangle + \langle S_3, C_3 \rangle + \langle S_3, C_4 \rangle) \\
 & = - \mu \langle S_3, C_2 + C_3 + C_4 \rangle = \mu \langle S_3, C_1 + C_5 \rangle \leq 0.
 \end{align*}
 
  \emph{Case \ref{item:N5Case6}.} 
\begin{align*}
 (\star)_5 &= - \mu (\langle S_2, C_2 \rangle + \langle S_3, C_3 \rangle + \langle S_4, C_4 \rangle + k_5 \langle S_5, C_5 \rangle) \\
 &\leq -\mu \langle S_4, C_2 + C_3 + C_4 + C_5 \rangle = \mu \langle S_4, C_1 \rangle \leq 0.
 \end{align*}
 
   \emph{Case \ref{item:N5Case7}.} 
\begin{align*}
 (\star)_5 &= - \mu (\langle S_2, C_2 \rangle + \langle S_3, C_3 \rangle + \langle S_4, C_4 \rangle + k_5 \langle S_5, C_5 \rangle) \\
 &\leq -\mu \langle S_5, C_2 + C_3 + C_4 + C_5 \rangle = \mu \langle S_5, C_1 \rangle \leq 0.
 \end{align*}
 
    \emph{Case \ref{item:N5Case8}.} 
\begin{align*}
 (\star)_5 &= - \mu (\langle S_2, C_2 \rangle + \langle S_3, C_3 \rangle + \langle S_4, C_4 \rangle + k_5 \langle S_5, C_5 \rangle) \\
 &\leq -\mu \langle S_3, C_2 + C_3 + C_4 + C_5 \rangle = \mu \langle S_3, C_1 \rangle \leq 0.
 \end{align*}
 
     \emph{Case \ref{item:N5Case9}.} 
\begin{align*}
 (\star)_5 &= - \mu (\langle S_2, C_2 \rangle + \langle S_3, C_3 \rangle + \langle S_4, C_4 \rangle + k_5 \langle S_5, C_5 \rangle) \\
 &\leq -\mu \langle S_3, C_2 + C_3 + C_4 + C_5 \rangle = \mu \langle S_3, C_1 \rangle \leq 0.
 \end{align*}
 
      \emph{Case \ref{item:N5Case10}.}
\begin{align*}
 (\star)_1 &= - \mu k_1 \langle S_1, C_1 \rangle + (1 - k_1) \langle T_1, C_1 \rangle + \langle T_2, C_2 \rangle + \langle T_3, C_3 \rangle + \langle T_4, C_4 \rangle \\
 &\leq \langle T_4, C_2 \rangle + \langle T_4, C_3 \rangle + \langle T_4, C_4 \rangle = \langle T_4, C_2 + C_3 + C_4 \rangle = - \langle T_4, C_1 + C_5 \rangle \leq 0.
 \end{align*}

       \emph{Case \ref{item:N5Case11}.} 
\begin{align*}
 (\star)_1 &= - \mu k_1 \langle S_1, C_1 \rangle + (1 - k_1) \langle T_1, C_1 \rangle + \langle T_2, C_2 \rangle + \langle T_3, C_3 \rangle + \langle T_4, C_4 \rangle + \langle T_5, C_5 \rangle \\
 &\leq \langle T_3, C_2 \rangle + \langle T_3, C_3 \rangle + \langle T_3, C_4 \rangle = \langle T_3, C_2 + C_3 + C_4 \rangle = - \langle T_3, C_1 + C_5 \rangle \leq 0.
\end{align*}
 
        \emph{Case \ref{item:N5Case12}.} 
\begin{align*}
 (\star)_4 & \le - \mu (\langle S_2, C_2 \rangle + \langle S_3, C_3 \rangle + k_4 \langle S_4, C_4 \rangle) \leq - \mu \langle S_2, C_2 + C_3 + C_4 \rangle \\
 &= - \mu \langle S_2, C_1 + C_2 + C_3 + C_4 \rangle \leq \mu \langle S_2, C_5 \rangle \leq 0
 \end{align*}
 
         \emph{Case \ref{item:N5Case13}.} 
\begin{align*}
 (\star)_2 &= - \mu k_2 \langle S_2 , C_2 \rangle + (1 - k_2) \langle T_2 , C_2 \rangle + \langle T_3 , C_3 \rangle + \langle T_4 , C_4 \rangle + \langle T_5 , C_5 \rangle \\
 &\leq \langle T_4 , C_3 + C_4 + C_5 \rangle = - \langle T_4 , C_1 + C_2 \rangle \leq 0.
 \end{align*}
\end{proof}

\bibliographystyle{plain}
\bibliography{ReferencesAcceptedVersion}

\begin{bibdiv}
\begin{biblist}

\bib{COR}{article}{
      author={Cordoba, Diego},
      author={Faraco, Daniel},
      author={Gancedo, Francisco},
       title={{Lack of Uniqueness for Weak Solutions of the Incompressible
  Porous Media Equation}},
        date={2010-10},
     journal={Archive for Rational Mechanics and Analysis},
      volume={200},
      number={3},
       pages={725\ndash 746},
}

\bib{GMTDI}{article}{
      author={De~Lellis, Camillo},
      author={De~Philippis, Guido},
      author={Kirchheim, Bernd},
      author={Tione, Riccardo},
       title={Geometric measure theory and differential inclusions},
    language={en},
        date={2021},
     journal={Annales de la Facult\'e des sciences de Toulouse :
  Math\'ematiques},
      volume={Ser. 6, 30},
      number={4},
       pages={899\ndash 960},
         url={https://afst.centre-mersenne.org/articles/10.5802/afst.1691/},
}

\bib{DIP}{article}{
      author={DiPerna, Ronald~J.},
       title={{Compensated Compactness and General Systems of Conservation
  Laws}},
        date={1985-12},
     journal={Transactions of the American Mathematical Society},
      volume={292},
      number={2},
       pages={383},
}

\bib{FS}{article}{
      author={F{\"o}rster, Clemens},
      author={Sz{\'e}kelyhidi, L{\'a}szl{\'o}},
       title={{$T_5$-Configurations and non-rigid sets of matrices}},
        date={2017-12},
        ISSN={1432-0835},
     journal={Calculus of Variations and Partial Differential Equations},
      volume={57},
      number={1},
       pages={19},
         url={https://doi.org/10.1007/s00526-017-1293-7},
}

\bib{HRT}{article}{
      author={Hirsch, Jonas},
      author={Tione, Riccardo},
       title={On the constancy theorem for anisotropic energies through
  differential inclusions},
        date={2021-04},
     journal={Calculus of Variations and Partial Differential Equations},
      volume={60},
      number={3},
}

\bib{KIRK}{misc}{
      author={Kirchheim, Bernd},
       title={{R}igidity and {G}eometry of {M}icrostructures},
        date={2003},
}

\bib{KMS}{misc}{
      author={Kirchheim, Bernd},
      author={M{\"u}ller, Stefan},
      author={{\v S}ver{\'a}k, Vladimir},
       title={{Studying Nonlinear PDE by Geometry in Matrix Space}},
        date={2003},
}

\bib{LP2}{article}{
      author={Lorent, Andrew},
      author={Peng, Guanying},
       title={{Null Lagrangian Measures in Subspaces, Compensated Compactness
  and Conservation Laws}},
        date={2019-07},
     journal={Archive for Rational Mechanics and Analysis},
      volume={234},
      number={2},
       pages={857\ndash 910},
}

\bib{LP1}{article}{
      author={Lorent, Andrew},
      author={Peng, Guanying},
       title={{On the Rank-1 convex hull of a set arising from a hyperbolic
  system of Lagrangian elasticity}},
        date={2020-08},
     journal={Calculus of Variations and Partial Differential Equations},
      volume={59},
      number={5},
}

\bib{DMU}{misc}{
      author={M{\"u}ller, Stefan},
       title={Variational {M}odels for {M}icrostructure and {P}hase
  {T}ransitions},
         how={Lectures at the CIME {S}ummer {S}chool {C}alculus of {V}ariations
  and {G}eometric {E}volution {P}roblems},
        date={1999},
}

\bib{SMVS}{article}{
      author={M{\"u}ller, Stefan},
      author={{\v S}ver{\'a}k, Vladimir},
       title={{C}onvex integration for {L}ipschitz mappings and counterexamples
  to regularity},
        date={2003},
        ISSN={0003-486X},
     journal={Annals of Mathematics},
      volume={157},
      number={3},
       pages={715\ndash 742},
      review={\MR{1983780}},
}

\bib{ST}{article}{
      author={Sorella, Massimo},
      author={Tione, Riccardo},
       title={The four-state problem and convex integration for linear
  differential operators},
        date={2023-02},
     journal={Journal of Functional Analysis},
      volume={284},
      number={4},
       pages={109785},
}

\bib{SAP}{article}{
      author={{\v S}ver{\'a}k, Vladimir},
       title={On {T}artar's conjecture},
        date={1993},
        ISSN={0294-1449},
     journal={Annales de l'Institut Henri Poincar\'{e}. Analyse Non
  Lin\'{e}aire},
      volume={10},
      number={4},
       pages={405\ndash 412},
      review={\MR{1246459}},
}

\bib{LSP}{article}{
      author={Sz{\'{e}}kelyhidi, L{\'{a}}szl{\'{o}}},
       title={The {R}egularity of {C}ritical {P}oints of {P}olyconvex
  {F}unctionals},
        date={2004-01},
        ISSN={0003-9527},
     journal={Archive for Rational Mechanics and Analysis},
      volume={172},
      number={1},
       pages={133\ndash 152},
}

\bib{SzeLas}{article}{
      author={Sz{\'e}kelyhidi, L{\'a}szl{\'o}},
       title={Rank-one convex hulls in $\mathbb{R}^{2 \times 2}$},
        date={2005/03/01},
     journal={Calculus of Variations and Partial Differential Equations},
      volume={22},
      number={3},
       pages={253\ndash 281},
         url={https://doi.org/10.1007/s00526-004-0272-y},
}

\bib{PER}{article}{
      author={Tione, Riccardo},
       title={Critical points of degenerate polyconvex energies},
        date={2022-03-23},
      eprint={2203.12284},
}

\end{biblist}
\end{bibdiv}

\end{document}